\documentclass[12pt,reqno]{amsart}
\usepackage{a4wide}
\usepackage[T1]{fontenc}
\usepackage{amssymb,amsmath,mathtools}
\usepackage{mathrsfs} 
\usepackage{enumitem} 
\usepackage[dvipsnames]{xcolor}
\usepackage{graphicx}
\usepackage{comment} 

\usepackage{hyperref}
\hypersetup{colorlinks=true, linkcolor=MidnightBlue, citecolor=teal, urlcolor=cyan}
\usepackage{amsrefs} 
\usepackage[nameinlink,capitalise]{cleveref}

\setlist[enumerate]{leftmargin=*, widest=iii}
\setlist[enumerate,1]{label=\rm{(\roman*)}, ref=\roman*}
\setlist[itemize]{leftmargin=*, widest=iii}

\newtheorem{theorem}{Theorem}[section]
\newtheorem{lemma}[theorem]{Lemma}
\AddToHook{env/lemma/begin}{\crefalias{theorem}{lemma}}
\newtheorem{corollary}[theorem]{Corollary}
\AddToHook{env/corollary/begin}{\crefalias{theorem}{corollary}}
\newtheorem{proposition}[theorem]{Proposition}
\AddToHook{env/proposition/begin}{\crefalias{theorem}{proposition}}
\newtheorem{fact}[theorem]{Fact}
\AddToHook{env/fact/begin}{\crefalias{theorem}{fact}}
\newtheorem{claim}[theorem]{Claim}
\AddToHook{env/claim/begin}{\crefalias{theorem}{claim}}

\AddToHook{env/problem/begin}{\crefalias{theorem}{problem}}

\newcounter{maintheorem}

\newtheorem{mainth}[maintheorem]{Theorem}
\AddToHook{env/mainth/begin}{\crefalias{maintheorem}{mainth}}

\theoremstyle{remark}
\newtheorem{remark}[theorem]{Remark}
\AddToHook{env/remark/begin}{\crefalias{theorem}{remark}}

\theoremstyle{definition}
\newtheorem{definition}[theorem]{Definition}
\AddToHook{env/definition/begin}{\crefalias{theorem}{definition}}

\AddToHook{env/example/begin}{\crefalias{theorem}{example}}

\numberwithin{equation}{section}
\makeatother

\newcommand{\R}{\mathbb{R}}
\newcommand{\ZZ}{\mathbb{Z}}
\newcommand{\N}{\mathbb{N}}
\newcommand{\e}{\varepsilon}
\renewcommand{\theta}{\vartheta}
\renewcommand{\rho}{\varrho}

\newcommand{\inte}{\mathrm{int}}
\newcommand{\cal}[1]{\mathcal{#1}}

\newcommand{\n}{\left\Vert\cdot\right\Vert}
\newcommand{\nn}[1]{{\left\vert\kern-0.25ex\left\vert\kern-0.25ex\left\vert #1 
\right\vert\kern-0.25ex\right\vert\kern-0.25ex\right\vert}}

\renewcommand{\leq}{\leqslant}
\renewcommand{\geq}{\geqslant}

\renewcommand{\epsilon}{\varepsilon}
\newcommand{\spn}{{\rm span}}

\DeclareMathOperator{\dens}{dens}
\DeclareMathOperator{\dist}{dist}

\newcommand{\X}{\mathcal{X}}
\newcommand{\Y}{\mathcal{Y}}
\newcommand{\Z}{\mathcal{Z}}
\newcommand{\C}{\mathcal{C}}
\newcommand{\F}{\mathcal{F}}
\newcommand{\B}{\mathcal{B}}
\renewcommand\qedsymbol{$\blacksquare$} 

\title[Tilings and coverings by balls in \texorpdfstring{$\ell_1$}{ell1}]{Tilings and coverings by balls in \texorpdfstring{$\ell_1$}{ell1}}

\author[C.A.~De~Bernardi]{Carlo Alberto De Bernardi}
\address[C.A.~De~Bernardi]{Dipartimento di Matematica per le Scienze economiche, finanziarie ed attuariali, Universit\`a Cattolica del Sacro Cuore, 20123 Milano, Italy\newline
\href{https://orcid.org/0000-0002-9654-1324}{ORCID: \texttt{0000-0002-9654-1324}}}
\email{carloalberto.debernardi@unicatt.it, carloalberto.debernardi@gmail.com}

\author[T.~Russo]{Tommaso Russo}
\address[T.~Russo]{Universit\"{a}t Innsbruck, Department of Mathematics, Technikerstra\ss e 13, 6020 Innsbruck, Austria \newline
\href{https://orcid.org/0000-0003-3940-2771}{ORCID: \texttt{0000-0003-3940-2771}}}
\email{tommaso.russo@uibk.ac.at, tommaso.russo.math@protonmail.com}

\author[\c S.~Sezgek]{\c Seyda~Sezgek}
\address[\c S.~Sezgek]{Department of Mathematics, Mersin University, Mersin, Turkey \newline
\href{https://orcid.org/0000-0001-9035-3114}{ORCID: \texttt{0000-0001-9035-3114}}}
\email{seydasezgek@mersin.edu.tr}

\author[J.~Somaglia]{Jacopo Somaglia}
\address[J.~Somaglia]{Politecnico di Milano, Dipartimento di Matematica, Piazza Leonardo da Vinci 32, 20133 Milano, Italy \newline
\href{https://orcid.org/0000-0003-0320-3025}{ORCID: \texttt{0000-0003-0320-3025}}}
\email{jacopo.somaglia@polimi.it}

\subjclass[2020]{Primary 46B04, 46B20, 51M20; Secondary 05B45, 52A05}
\keywords{Tiling by balls, Covering, Star-finite covering}

\begin{document}
\begin{abstract} A famous result of Klee from 1981 is that the Banach space $\ell_1(\kappa)$ admits a disjoint tiling by balls of radius $1$, for all cardinals $\kappa$ with $\kappa^{\aleph_0} =\kappa$. Klee also observed that the smallest cardinal in which such a tiling might exist is $\kappa= 2^{\aleph_0}$, leaving open the question whether, for $\kappa< 2^{\aleph_0}$, $\ell_1(\kappa)$ might admit a tiling by balls at all. Our main result answers this question in the negative, proving in particular that $\ell_1$ does not admit any tiling by balls. We also give a companion result about star-$n$-finite coverings by balls of $\ell_1(\kappa)$ and we give a construction of a star-finite tiling of $\mathcal{X} \oplus_\infty c_{00}$, for each space $\mathcal{X}$ whose dimension is at most countable.
\end{abstract}
\maketitle

\section{Introduction}
One of the most important results on tilings of infinite-dimensional spaces is Klee's celebrated construction of a disjoint tiling of $\ell_1(\kappa)$ by balls of radius $1$, for each cardinal $\kappa$ such that $\kappa^{\aleph_0} =\kappa$, \cite{Klee_MathAnn}. This construction has been revisited, modified, or generalised in several papers, for instance, \cites{DRSS_packing, DRS_hilbert, DEVEtiling, Klee_nice_tilings, KleeMalZan, Swanepoel}. Our paper is also motivated by this result and focuses on the role of the cardinal $\kappa$. In fact, Klee himself noted that, when $\kappa< 2^{\aleph_0}$, $\ell_1(\kappa)$ does not admit any \emph{disjoint} tiling by balls of radius $1$, \cite{Klee_MathAnn}*{Theorem 3.5}. This result left open the question whether said $\ell_1(\kappa)$ spaces might admit any tiling by balls at all, which we study (and answer in the negative) in our paper. For further information on tilings in infinite dimensions, we refer to \cites{DEBEpointfinite, DRS_hilbert, DEVEtiling, FPZ, FZ_PAMS06, Preiss} and references therein.

There are several results in the literature that preclude the existence of a tiling by balls (or, more generally, by bodies), based on some smoothness/rotundity properties. For instance, no separable normed space can admit a tiling by rotund bodies, \cite{KleeMalZan}; likewise, separable Banach spaces do not admit tilings by smooth and bounded convex bodies, \cite{KleeTri}. Without the separability assumption, it is proved in \cite{DEVEtiling} that Fr\'echet smooth or LUR Banach spaces do not admit tilings by balls. In the same paper it is also shown that normed spaces whose unit ball has some LUR point don't admit tilings by balls with radii bounded below. However, it is apparent that none of the above results applies to $\ell_1(\kappa)$.

Another related problem on which we provide some, this time partial, contribution is the study of star-finite coverings of Banach spaces, \cites{Breen, DESOVEstar, Nielsen}. A covering $\C$ of a normed space $\X$ is star-finite if each set in $\C$ intersects only finitely many elements of $\C$. Vesel\'y and the first- and fourth-named authors proved in \cite{DESOVEstar} that LUR, or uniformly Fr\'echet smooth, infinite-dimensional Banach spaces do not admit star-finite coverings by balls. It is also proved there that, for infinite $\kappa$, the space $c_0(\kappa)$ does not admit a star-finite covering by balls either. Since there is no known example of a separable Banach space admitting such a covering, the next space that is natural to consider after the results in \cite{DESOVEstar} is $\ell_1$. For the statement of our result in this direction we need one definition. A covering $\C$ of a normed space is \emph{star-$n$-finite} ($n\in \N$) if each set $C\in \C$ intersects at most $n$ sets of $\C\setminus \{C\}$. Let us now state the main result of our paper.

\begin{mainth}\label{mth: ell1} Let $\kappa$ be an infinite cardinal with $\kappa<2^{\aleph_0}$ and $n\in \N$. Then:
\begin{enumerate}
    \item\label{i: mth ell1} $\ell_1(\kappa)$ does not admit tilings by balls;
    \item $\ell_1(\kappa)$ does not admit star-$n$-finite coverings by balls (of positive radius).
\end{enumerate}
\end{mainth}

The first clause of the theorem generalises Klee's result quoted at the beginning, while the second one is the promised partial answer to the problem of existence of star-finite coverings by balls in $\ell_1$. Together with the results in \cite{DESOVEstar}, this second part seems to suggest that no infinite-dimensional separable normed space could admit a star-$n$-finite covering by balls. A consequence of \Cref{mth: ell1}, in the spirit of \cite{Klee_MathAnn}*{Theorem 3.5}, is that the continuum hypothesis \textsf{CH} is equivalent to the fact that $\ell_1(\aleph_1)$ admits a tiling by balls (\Cref{cor: CH iff}). The proof of \Cref{mth: ell1}, that we give in \Cref{sec: ell1}, also gives some slightly more general results, some of which require completeness in their proof. Our second main result was partially motivated by the question whether completeness could be omitted in their statements (see \Cref{rmk: noncomplete false}).

\begin{mainth}\label{mth: c0} Let $\X$ be a separable normed space. Then:
\begin{enumerate}
    \item\label{i: mth c0} the space $\X\oplus_\infty c_0$ admits a tiling by balls if and only if $\X$ does;
    \item\label{i: mth c00} if $\X$ has at most countable dimension, $\X\oplus_\infty c_{00}$ admits a star-finite tiling by balls.
\end{enumerate}
\end{mainth}

The proof of \Cref{mth: c0} will be given in \Cref{sec: c0}. The first part follows directly from the more general results of \Cref{sec: ell1} mentioned above, while the second part shows that $c_0$ cannot be replaced by $c_{00}$ in the former. In particular, \eqref{i: mth c00} implies that $c_{00}$ admits a star-finite tiling by balls and gives the first example of an infinite-dimensional separable normed space admitting such a tiling (the results in \cite{DESOVEstar} only give star-finite \emph{coverings}). Notice also that one cannot obtain a star-finite tiling in \eqref{i: mth c0}, since $c_0$ itself doesn't even have star-finite coverings by balls.

In conclusion, we mention one connection of \Cref{mth: ell1} with packing problems, \cites{DRSS_packing, Swanepoel} (and we refer to these papers for unexplained notions that we use in this paragraph). It follows from results in \cites{DRSS_packing, DRS_hilbert} that there exists a collection of non-overlapping balls of radius $1$ in $\ell_1$ whose union is dense in $\ell_1$ (and whose centers are a subgroup). Specifically, the result is proved for $\ell_1(\kappa)$, for all infinite cardinals $\kappa$, in \cite{DRS_hilbert}*{Corollary 3.4}, while \cite{DRSS_packing}*{Section~4.1} contains a more general result that applies in particular to all separable octahedral normed spaces. In terms of the lattice simultaneous packing and covering constant, this result implies that $\gamma^*(\ell_1)=1$. For finite-dimensional spaces, a standard compactness argument shows that $\gamma^*(\X)=1$ if and only if $\X$ admits a lattice tiling by balls of radius $1$ and it remained open whether the same was true for infinite-dimensional normed spaces as well. Thus, \Cref{mth: ell1}\eqref{i: mth ell1} also answers this problem in the negative.

\section{Preliminaries}\label{sec: prelim}
If $\X$ is a non-trivial real normed space then  $B_{\X}$, $U_\X$, and $S_{\X}$ denote the closed unit ball, the open unit ball, and the unit sphere of $\X$, respectively. Moreover, we denote by $B(x,\e)$, or occasionally $B_\X(x,\e)$, the closed ball in $\X$ with radius $\e\geq 0$ and centre $x$. When $\e=0$, we say that $B(x,0)= \{x\}$ is a \emph{degenerate ball}. In general, by a \emph{ball} in $\X$ we mean a (possibly degenerate) closed ball in $\X$. If $B\subseteq \X$ is a ball then  $c(B)$ and $r(B)$ denote its centre and radius, respectively.

Let $\F$ be a family of non-empty sets in a normed space $\X$. By $\bigcup \F$ we mean the union of all members of $\F$. For $x\in \X$ we denote
\[ \F_x\coloneqq \{F\in \F\colon x\in F\}. \]
Thus, $\F$ is a \emph{covering} of $\X$ if and only if $\F_x\neq \emptyset$ for each $x\in \X$ if and only if $\bigcup \F=\X$. A family $\F$ is \emph{non-overlapping} if the elements of $\F$ have mutually disjoint interiors. The family $\F$ is a \emph{tiling} if it is a non-overlapping covering and each element of $\F$ is a closed convex set with non-empty interior. In all the paper, we will only consider families $\F$ of (possibly degenerate) balls. Let us explicitly point the reader's attention to one subtlety: since we allow degenerate balls, a non-overlapping covering is in general not a tiling (this deviates from the standard notation, but for us it is a necessary distinction).

Let us also record explicitly the notions that we already mentioned in the Introduction. A family $\F$ is \emph{star-finite} if each element of $\F$ intersects only finitely many elements of the family $\F$. For $n\in \N_0$, the family is \emph{star-$n$-finite} if each $F\in \F$ intersects at most $n$ elements of $\F\setminus \{F\}$. For instance, star-$0$-finite means that the sets in $\F$ are mutually disjoint.
	
Given a covering $\F$ for $\X$, we say that $\F$ is \emph{minimal} if no proper subfamily of $\F$ is a covering for $\X$. While general coverings might not admit minimal subcovers (consider, \emph{e.g.}, $\{n\cdot B_\X\}_{n\in \N}$), we will need the fact that every star-finite covering admits a minimal subcover (\cite{DESOVEstar}*{p.~3}). Likewise, it is easy to check that a non-overlapping covering by balls admits a minimal subcover (take all non-degenerate balls and all singletons that are not contained in the non-degenerate balls).

A point $x\in \X$ is a \emph{regular point} for $\F$ if it admits a neighbourhood that intersects only finitely many members of $\F$. The family $\F$ is \emph{locally finite} if each $x\in \X$ is a regular point for $\F$. Points that are not regular are called \emph{singular}. We will denote by $S_\F$ the collection of singular points for $\F$; notice that it is a closed set. Finally, we say that $\F$ is \emph{protective for $x$} if $x\in \inte \bigl(\bigcup \F_x\bigr)$.

We conclude this section with a couple of simple remarks on singular points. Suppose that $\F$ is a covering of $\X$, and $x\in \X$. If, for some $F\in \F$, $x\in \partial F$ and $\F_x= \{F\}$, then $x$ is a singular point for $\F$ (this follows, \emph{e.g.}, from a small modification of \Cref{fact: non-protect point} below, or of \cite{Klee_nice_tilings}*{Theorem~1.1}). As a consequence, if $\F$ is a minimal covering, all singletons in $\F$ are singular points. Moreover, if the covering $\F$ is star-finite or non-overlapping, no singular point can be in the interior of some $F\in \F$.

\section{Proof of \texorpdfstring{\Cref{mth: ell1}}{Theorem A}}\label{sec: ell1}
The proofs of our two results concerning tilings and coverings follow the same general pattern, and we will prove the two results at once. Both require first a result in finite dimensions (\Cref{prop: finite dim}), based on applying the Baire category theorem to the set of singular points, and then a finite-dimensional reduction to deduce the main result (\Cref{thm: fd reduction}). This finite-dimensional reduction is somewhat similar to the arguments in \cite{DEBEpointfinite}*{Theorem~3.3} and \cite{Klee_MathAnn}*{Proposition~3.3}. When applying the Baire category theorem to the set of singular points, we need a sufficiently large supply of them. The method to obtain singular points is different in the two cases, thus we split the results we need into the first two subsections.

\subsection{Protectable points} 
Loosely speaking, the first result we need is that it is impossible to construct local tilings by regular octahedra in $\R^3$. To make this precise, we introduce the following definition, inspired by \cite{Klee_nice_tilings}*{Section 1}.

\begin{definition} We say that $x\in S_{\X}$ is \emph{protectable} if there exists a family $\B$ of non-overlapping balls containing $x$ such that $B_\X \in\B$ and $\B$ is protective for $x$. 
\end{definition}

In other words, the definition requires that the family $\B$ contains $B_\X$, consists of non-overlapping balls containing $x$, and its union is a neighbourhood of $x$. Notice that, since only balls containing $x$ are taken into consideration and the point $x$ itself is already covered by $B_\X$, we can suppose that all balls in $\B$ have positive radius.

\begin{fact}\label{fact: non-protect point} Suppose that a point $x\in S_\X$ is not protectable. Then, for every non-overlapping covering $\B$ of $\X$ by balls, with $B_\X\in \B$, $x$ is a singular point for $\B$. As a consequence:
\begin{enumerate}
    \item\label{item: 2 sing pts} If $S_\X$ admits a non-protectable point, each non-singleton element of a covering $\B$ of $\X$ by non-overlapping balls admits at least two antipodal singular points.
    \item\label{item: loc fin => protect} If $\X$ admits a non-overlapping and locally finite covering by balls, every point in $S_\X$ is protectable.
\end{enumerate}
\end{fact}

\begin{proof} Towards a contradiction, suppose that there is a non-overlapping covering $\B$ of $\X$, with $B_\X \in \B$, such that $x$ is a regular point for $\B$. Thus, there are a neighbourhood $U$ of $x$ and balls $B_1,\dots, B_n\in \B$ such that $U\subseteq B_1\cup\dots \cup B_n$. Since these balls are closed, $U'\coloneqq U\setminus \bigcup\{B_j\colon j=1,\dots, n,\, x\notin B_j\}$ is a neighbourhood of $x$ as well. Finally, the family $\B'\coloneqq \{ B_j\colon j=1,\dots, n,\, x\in B_j \}$ is protective for $x$ (as $U'\subseteq\bigcup \B'$), a contradiction.

To deduce \eqref{item: 2 sing pts}, take any covering of $\X$ by non-overlapping balls, and a ball $B\in \B$ with positive radius. Up to translating and scaling $\B$, we can suppose that $B= B_\X$. If $x\in S_\X$ is non-protectable, by symmetry and the first part, $\pm x$ are singular points for $\B$.

Finally, for \eqref{item: loc fin => protect} it is enough to notice that if a covering $\B$ is locally finite, then it must contain at least one ball of positive radius. The claim then follows from \eqref{item: 2 sing pts}, as $\B$ doesn't have singular points.
\end{proof}

We will also need the fact that the regular octahedron in $\R^3$ has some non-protectable points; as it turns out, we can describe exactly all protectable points.
\begin{lemma}\label{lemma: octahedron} Let $\X=(\R^3,\n_1)$. Then, $x\in S_{\X}$ is protectable if and only if it belongs to the relative interior of  a face of the regular octahedron $B_\X$.
\end{lemma}

\begin{proof} If $x$ belongs to the relative interior of a face then it is clearly protectable. Let us prove the other implication. Suppose that  there exists a family $\B$ of non-overlapping balls containing $x$ such that $B_\X\in\B$ and $\B$ is protective for $x$. Without loss of generality, we assume that all balls in $\B$ have positive radius. For every $B\in\B$, let $\theta_B$ denote the solid angle subtended by $x$ relative to the octahedron $B$. Then, $\sum_{B\in\B}\theta_B$ coincides with  $4\pi$ (the measure of the full solid angle). 
    
It is a standard fact (see, \emph{e.g.}, \cite{coxeter}*{p.~293}) that the solid angle subtended by a boundary point $x_0$ of a regular octahedron is equal to: \begin{itemize}
    \item $\omega_V\coloneqq 4\arcsin(\frac13)$, if $x_0$ is a vertex; 
    \item $\omega_E\coloneqq 2\arccos(-\frac13)= 2\arcsin(\frac13)+ \pi$, if $x_0$ is in the relative interior of an edge;
    \item $\omega_F\coloneqq 2\pi$, if $x_0$ is in the relative interior of a face.
    \end{itemize}
Since $\theta_B\in \{\omega_V,\omega_E,\omega_F\}$, there exist non-negative integers $a_V, a_E, a_F$ such that
\begin{eqnarray*}
    4\pi= \sum_{B\in\B}\theta_B&=& a_V\omega_V+ a_E\omega_E+ a_F\omega_F\\
    &=& 2(2a_V+a_E)\arcsin(\tfrac13)+(2a_F+a_E)\pi.
\end{eqnarray*}
Finally, $\arcsin(\frac13)$ is not a rational multiple of $\pi$ by Niven's theorem (see, \emph{e.g.}, \cite{Niven}*{Corollary~3.12}); therefore, we necessarily have that $a_V= a_E= 0$ and $a_F=2$, proving that $x$ belongs to the relative interior of a face of $B_\X$, as desired.
\end{proof}

\subsection{Star-\texorpdfstring{$n$}{n}-finite coverings}
The next result we need is that star-$n$-finite coverings in $n$-dimensional spaces necessarily admit many singular points.

\begin{lemma}\label{lemma: sing point in dim n} Let $\B$ be a minimal covering by balls of a normed space $\X$ of dimension $n$. If $\B$ is star-$n$-finite, then every element of $\B$ contains a singular point. In particular, if $\B$ is locally finite, it is not star-$n$-finite.
\end{lemma}

\begin{proof} To begin with, we notice the following fact. Given a normed space $\X$ of dimension $n$, a bounded convex body $C$ in $\X$, and closed convex sets $B_1,\dots, B_n$ in $\X$ such that $\partial C\subseteq B_1\cup\dots\cup B_n$, then $C\subseteq B_1\cup\dots \cup B_n$. We prove this by induction on $n$, the case $n=1$ being clear. Suppose that the result holds for some $n- 1$ and, towards a contradiction, that it fails to hold for $n$. Hence, there are a normed space $\X$ of dimension $n$, a bounded convex body $C$ in $\X$, and closed convex sets $B_1, \dots ,B_n$ such that $\partial C\subseteq B_1\cup\dots\cup B_n$, but $C$ is not contained in $B_1\cup\dots\cup B_n$. Take $x_0\in \inte(C) \setminus [B_1\cup \dots\cup B_n]$ and find a hyperplane $H$ such that $x_0\in H$ and $H\cap B_n= \emptyset$. Let $C'\coloneqq C\cap H$ and $B_k'\coloneqq B_k\cap H$ ($k=1,\dots, n-1$). Then, $C'$ is a body in $H$, a normed space of dimension $n-1$. Further,
\[ \partial C'\subseteq (\partial C)\cap H \subseteq B_1'\cup \ldots \cup B_{n-1}', \]
and $x_0\in C'\setminus( B_1'\cup \ldots \cup B_{n-1}')$. This contradicts the inductive assumption.

We now prove the claimed result. It is enough to prove that, for each $B\in \B$, there exists a point $x\in \partial B$ such that $\B_{x}=\{B\}$. In fact, $\B$ being a covering, this implies that $x$ is singular. If $B=\{x\}$ is a singleton, then $\B_{x}=\{B\}$ holds just by minimality of $\B$. Thus, we can assume that $B$ is a body. If $\partial B$ were covered by elements of $\B\setminus \{B\}$, by star-$n$-finiteness, there would be balls $B_1,\dots, B_n$ with $\partial B\subseteq B_1\cup\dots \cup B_n$. By the fact at the beginning of the proof, $B\subseteq B_1\cup\dots \cup B_n$, which contradicts minimality.
\end{proof}

\begin{remark} The fact at the beginning of the proof is a folklore observation, see, \emph{e.g.}, \cite{MP09}*{Lemma 1.1} or \cite{Papini}*{Proposition 3bis}. Incidentally, the argument above is valid for every bounded set $C$. In fact, if $C$ has non-empty interior, the argument above applies verbatim, while otherwise the statement is tautologically true.
\end{remark}

\subsection{Main results}
We can now prove our main results, and we begin with the crucial finite-dimensional step.

\begin{proposition}\label{prop: finite dim} Let $\X$ be a finite-dimensional normed space and let $\B$ be a countable covering by balls of $\X$.
\begin{enumerate}
    \item\label{item: non-overlapping implies protectable} If $\B$ is non-overlapping, every point in $S_{\X}$ is protectable.
    \item\label{item: dim=n not star-n-finite} If $n\coloneqq \dim \X$, then $\B$ is not star-$n$-finite.
\end{enumerate} 
\end{proposition}

Notice that the assumption that $\B$ is countable is essential, as every finite-dimensional normed space can be written as a disjoint union of $2^{\aleph_0}$ degenerate balls.

\begin{proof} The first part of the proof is common for \eqref{item: non-overlapping implies protectable} and \eqref{item: dim=n not star-n-finite}. If $\B$ is locally finite, then the two items follow from \Cref{fact: non-protect point}\eqref{item: loc fin => protect} and \Cref{lemma: sing point in dim n}, respectively. Hence, we assume that $\B$ is not locally finite. Further, in \eqref{item: dim=n not star-n-finite} we can assume that $\B$ is star-finite, because otherwise it would not be star-$n$-finite either. Thus, in either case, we can also assume that $\B$ is minimal.

As a consequence, the set $S_\B$ of singular points is non-empty (and closed). As we mentioned in \Cref{sec: prelim}, the singular points of a non-overlapping or star-finite covering cannot be in the interior of any body of the covering, thus
\[ S_\B= \bigcup_{B\in \B} (S_\B\cap \partial B). \]
By the Baire category theorem there exist $B_0\in\B$, $x_0\in S_\B\cap\partial B_0$, and $\e>0$ such that
\begin{equation}\label{eq: Baire}
    S_\B\cap B(x_0,\e)\subseteq\partial B_0.
\end{equation}
Since $x_0$ is a singular point, there exists a sequence $(C_k)_{k\in\N}$ of pairwise distinct elements of $\B$ such that $C_k\cap B(x_0, \e/2)\neq \emptyset$ for each $k\in \N$.
\begin{claim}\label{claim: radius to 0} $\lim_{k\to \infty} r(C_k)=0$, hence $C_k\subseteq B(x_0,\e)$ for all sufficiently large $k\in \N$.
\end{claim}

\begin{proof}[Proof of \Cref{claim: radius to 0}]\renewcommand\qedsymbol{$\square$}
If the conclusion is false, up to passing to a subsequence, we can assume that $r(C_k)\geq \delta>0$ for all $k\in \N$. Further, in case \eqref{item: dim=n not star-n-finite}, we can pass to a further subsequence and assume that the $C_k$'s are mutually disjoint; thus, we assume that $\{C_k\}_{k\in\N}$ is non-overlapping. Take balls $C'_k\subseteq C_k$ with $r(C'_k)= \delta$ and $C'_k\cap B(x_0,\e/2) \neq \emptyset$. Then, the ball $B(x_0, 2\delta+ \e/2)$ contains a sequence of non-overlapping balls of the same radius $\delta>0$, which is clearly impossible as $\X$ is finite dimensional.
\end{proof}

To prove \eqref{item: non-overlapping implies protectable}, assume $\B$ is non-overlapping and suppose, towards a contradiction, that there exists a point in $S_\X$ that is non-protectable. Let $k\in\N$ be such that $C_k \subseteq B(x_0,\e)$. If $C_k= \{c\}$, then by minimality of $\B$, $c$ is a singular point for $\B$ and $c\notin B_0$. If $r(C_k)>0$ then, by \Cref{fact: non-protect point}\eqref{item: 2 sing pts}, there exist two antipodal points in  $C_k \cap S_\B$. A simple geometric argument shows that at least one point does not belong to $B_0$ (otherwise, $C_k$ and $B_0$ would overlap). In any case, we get a contradiction by \eqref{eq: Baire}, and \eqref{item: non-overlapping implies protectable} is proved.

Finally, for \eqref{item: dim=n not star-n-finite}, assume by contradiction that $\B$ is star-$n$-finite and find $k\in \N$ so that $C_k\subseteq B(x_0,\e)$ and $C_k\cap B_0=\emptyset$. By \Cref{lemma: sing point in dim n}, $C_k$ contains a singular point, which contradicts \eqref{eq: Baire}.
\end{proof}

Finally, we can move to our infinite-dimensional results. In order to perform the finite-dimensional reduction, we need subspaces with the following property.

\begin{definition}\label{def: good subspace} Let $\X$ be a normed space. A subspace $\Y\subseteq \X$ has the \emph{good intersection of balls} property (property (GIB) for short) if the intersection of every ball of $\X$ with $\Y$ is either empty or a (possibly degenerate) ball in $\Y$.
\end{definition}

\begin{remark}\label{rmk: GIB} For example, if $I\subseteq \kappa$, the subspace $\ell_1(I)$ of $\ell_1(\kappa)$ has the (GIB). In fact, for $x\in \ell_1(\kappa)$, write $x=(x_1, x_2)\in \ell_1(I)\oplus_1 \ell_1(\kappa\setminus I)$. Then for all $r\geq \|x_2\|$ we have
\begin{equation}\label{eq: GIB for ell1}
    B(x,r)\cap \ell_1(I)= B(x_1,r- \|x_2\|)\cap \ell_1(I).
\end{equation}
A similar computation also works for the spaces $\ell_p(\kappa)$, $1\leq p\leq \infty$, and $c_0(\kappa)$. For later use, we notice that it also applies to the spaces of finitely supported sequences, as $c_{00}(\kappa)$, or the linear spans of the canonical basis in $\ell_p(\kappa)$, $1\leq p\leq \infty$.
\end{remark}

\begin{fact}\label{fact: good intersection with balls} Let $\Y$ be a proper closed subspace of a normed space $\X$ such that $\X/ \Y$ is complete and let $\B$ be a countable family of balls in $\X$. Then, there exists $x_0\in \X$ such that, for all $B\in\B$ satisfying $(x_0+B)\cap \Y\neq \emptyset$, we have
\[ (x_0+ \inte B)\cap \Y\neq \emptyset. \]
In particular, if $\B$ is non-overlapping, then the same holds for the family
\[ \{(x_0+B)\cap \Y\colon B\in \B\}. \]
\end{fact}

\begin{remark} It is worth pointing out that this fact is in general false in absence of completeness. In fact, taking $\Y=\{0\}$, the conclusion becomes: there exists $x_0\in \X$ such that, for all $B\in\B$, $0\in x_0+ B$ if and only if $0\in x_0+ \inte B$. Equivalently, there exists a point $x_0$ that doesn't belong to the boundary of any ball in $\B$. This condition does not hold in general, as it is possible that a normed space is the union of countably many spheres. To wit, consider the space $c_{00}$ and the spheres of radius $1$ centred at the points of the form $(k_1,\dots,k_n,0,\dots)$, where $k_1,\dots,k_n$ are odd integers.

The fact might also be false if the family $\B$ is uncountable. In fact, similarly as above, the Banach space $c_0(\aleph_1)$ is the union of $\aleph_1$ spheres (of radius $1$). For this, it is enough to use as centres points of the form $(x(\alpha))_{\alpha<\aleph_1}$ where each $x(\alpha)$ is either an odd integer or $0$.
\end{remark}

\begin{proof}[Proof of \Cref{fact: good intersection with balls}] Denote $\Z\coloneqq \X/ \Y$ and let $q\colon \X\to \Z$ be the canonical quotient map. If $B$ is a ball in $\X$ with positive radius, then (recall that $U_\Z= \inte B_\Z$)
\[ q(B)\subseteq q\left(c(B)\right)+r(B) B_\Z\quad \text{ and}\quad q(\inte B)=q\left(c(B)\right)+r(B) U_\Z. \]
In particular, if $x\in \X $ is such that $(x+B)\cap \Y\neq \emptyset$ and $(x+ \inte B)\cap \Y= \emptyset$, then
\begin{equation}\label{eq: project a ball}
    -q(x)\in q\left(c(B)\right)+r(B) S_\Z.
\end{equation} 
Note that the last implication holds trivially even if $B$ is a degenerate ball. Now, suppose on the contrary that our conclusion does not hold, that is, for every $x\in \X$ there exists $B\in \B$ such that $(x+B)\cap \Y\neq \emptyset$ and 
$ (x+ \inte B)\cap \Y= \emptyset$. Our previous observation implies that
\[\Z =\bigcup_{x\in \X}-q(x)=\bigcup_{B\in\B}q\left(c(B)\right)+r(B) S_\Z,\]
namely, $\Z$ can be written as a countable union of spheres (recall that $\B$ is countable by assumption). However, by the Baire category theorem this is a contradiction, as every sphere in $\Z$ is nowhere dense and $\Z$ is complete.
\end{proof}

\begin{theorem}\label{thm: fd reduction} Let $\X$ be a normed space with $\dens(\X)<2^{\aleph_0}$ and $\Y$ be a finite-dimensional subspace of $\X$ satisfying property (GIB). 
\begin{enumerate}
    \item\label{tiling sep} If $\X$ is a separable Banach space and $S_{\Y}$ contains a non-protectable point, then $\X$ does not admit tilings by balls.
    \item\label{tiling ell1} If $3\leq \kappa< 2^{\aleph_0}$, the space $\X= \ell_1(\kappa)$ does not admit tilings by balls.
    \item\label{star-n-finite} If the dimension of $\Y$ is at least $n$, then $\X$ does not admit star-$n$-finite coverings by closed balls, each of positive radius.
\end{enumerate}
\end{theorem}

As we will see in \Cref{rmk: noncomplete false}, the assumption of completeness cannot be omitted in \eqref{tiling sep}.

\begin{proof} As in the proof of \Cref{prop: finite dim}, the first part is common to all the clauses. Arguing by contradiction, we take a covering $\B$ of $\X$ consisting of balls of positive radius. Further, in \eqref{tiling sep} and \eqref{tiling ell1} we assume that $\B$ is a tiling and in \eqref{star-n-finite} that it is a star-$n$-finite covering. Moreover, in the case of $\ell_1(\kappa)$, the subspace $\Y$ is the subspace $\ell_1^3$, spanned by three vectors of the canonical basis (here we are using the assumption that $\kappa\geq 3$); recall that, by \Cref{rmk: GIB}, $\Y$ has the (GIB) property. Let us denote
\[ \B'\coloneqq \{B\cap \Y\colon B\in\B,\ B\cap \Y\neq\emptyset \}, \]
which, by the (GIB) property, is a covering of $\Y$ with (possibly degenerate) balls.

To begin with, $\B'$ retains some properties of $\B$. In fact, it is clear that in \eqref{star-n-finite} $\B'$ is a star-$n$-finite covering. Likewise, we claim that in \eqref{tiling sep} and \eqref{tiling ell1} the family $\B'$ consists of non-overlapping balls. In fact, in \eqref{tiling sep}, as $\B$ is countable, we can apply \Cref{fact: good intersection with balls} and, up to translating each ball in $\B$ by the same vector, we assume without loss of generality that each ball $B$ in $\B$ that intersects $\Y$ satisfies $\Y\cap \inte B\neq \emptyset$; as a consequence, the balls in $\B'$ are non-overlapping. For \eqref{tiling ell1}, we argue similarly, but we replace \Cref{fact: good intersection with balls} with the stronger property \eqref{eq: GIB for ell1} of $\ell_1(\kappa)$ spaces. In particular, \eqref{eq: GIB for ell1} implies that, if $r(B\cap \Y)>0$, then $\Y\cap \inte B\neq \emptyset$; this yields that $\B'$ is non-overlapping.

We now decompose the family $\B'$ into non-degenerate and degenerate balls:
\[ \cal C'\coloneqq \{B\in\B'\colon r(B)>0\} \quad\text{and}\quad \cal D'\coloneqq \{B\in\B'\colon |B|=1\}. \]
Observe that $|\B|\leq\dens(\X)<2^{\aleph_0}$ and $|\cal C'|\leq \aleph_0$. In fact, in \eqref{tiling sep} and \eqref{tiling ell1} this is obvious as the sets in $\B$, resp.~$\cal C'$, have disjoint non-empty interiors. Instead, in \eqref{star-n-finite} it directly follows from \cite{DESOVEstar}*{Lemma~2.2}. Since $\cal C'$ is at most countable, we deduce that $\Y\setminus\bigcup \cal C'$ is a Borel subset (actually, $G_\delta$) of $\Y$. Further, $\Y\setminus\bigcup \cal C'$ is contained in $\bigcup \cal D'$, hence
\[ |\Y\setminus\bigcup \cal C'|\leq|\bigcup \cal D'|=|\cal D'|\leq|\B|<2^{\aleph_0}. \]

Recalling that Borel subsets of Polish spaces have the perfect-set property (see, \emph{e.g.}, \cite{Kechris}*{Theorem~13.6}), we conclude that $\Y\setminus\bigcup \cal C'$ must be countable. Hence, the family
\[ \B''\coloneqq \cal C'\cup\bigl\{\{y\};\, y\in \Y\setminus\bigcup \cal C'\bigr\} \]
is a countable covering of $\Y$ by balls.

We now divide the various cases. To begin with, in the case of \eqref{tiling sep}, $\B''$ is non-overlapping because $\B'$ was. Hence, by \Cref{prop: finite dim}\eqref{item: non-overlapping implies protectable} each point of $S_\Y$ is protectable, a contradiction. Similarly, in \eqref{tiling ell1}, we obtain that $\B''$ is non-overlapping and that every point of the unit sphere $S_\Y$ is protectable. However, this contradicts \Cref{lemma: octahedron}. Finally, in case \eqref{star-n-finite}, it is clear that $\B''$ is a star-$n$-finite covering of $\Y$. Since $\dim\Y \geq n$, this contradicts \Cref{prop: finite dim}\eqref{item: dim=n not star-n-finite} and concludes the proof.
\end{proof}

For the sake of completeness, we restate explicitly our main results in the case of $\ell_1(\kappa)$.
\begin{corollary} Let $\kappa$ be an infinite cardinal with $\kappa<2^{\aleph_0}$ and $n\in \N$. Then:
\begin{enumerate}
    \item $\ell_1(\kappa)$ does not admit tilings by balls;
    \item $\ell_1(\kappa)$ does not admit star-$n$-finite coverings by balls, each of positive radius.
\end{enumerate} 
\end{corollary}

\begin{remark}\label{rmk: star not star-n} As \Cref{thm: fd reduction}\eqref{star-n-finite} does not require completeness, it also follows that the space $\ell_{1,0}$ of finitely supported sequences does not admit star-$n$-finite coverings by balls of positive radius (for any $n\in \N$). On the other hand, this space has a countable (algebraic) basis, so it admits a star-finite covering by balls of positive radius, \cite{DESOVEstar}*{Theorem~2.12}. In particular, the property of being star-finite is weaker than being star-$n$-finite for some $n\in \N$. Moreover, the proof of \Cref{thm: fd reduction}\eqref{tiling ell1} is also independent of completeness, hence it applies to $\ell_{1,0}(\kappa)$ as well ($\kappa< 2^{\aleph_0}$). Thus, it also holds that, when $\kappa< 2^{\aleph_0}$, $\ell_{1,0}(\kappa)$ does not admit tilings by closed balls.
\end{remark}

We can also obtain the following consequence, that gives a reformulation of the continuum hypothesis in terms of tilings of $\ell_1(\aleph_1)$ by balls. This is in the same spirit of Klee's result \cite{Klee_MathAnn}*{Theorem 3.5} that \textsf{CH} holds if and only if $\ell_1(\aleph_1)$ admits a disjoint tiling by balls of radius $1$.

\begin{corollary}\label{cor: CH iff} The following conditions are equivalent:
\begin{enumerate}
    \item\label{i: CH} The continuum hypothesis \textup{\textsf{CH}} holds, namely $\aleph_1= 2^{\aleph_0}$;
    \item\label{i: Klee} $\ell_1(\aleph_1)$ admits a disjoint tiling by balls of radius $1$;
    \item\label{i: DRS} $\ell_1(\aleph_1)$ admits a lattice tiling by balls;
    \item\label{i: tiling} $\ell_1(\aleph_1)$ admits a tiling by balls.
\end{enumerate}
\end{corollary}

\begin{proof} The implication \eqref{i: CH}$\implies$\eqref{i: Klee} is just Klee's original construction \cite{Klee_MathAnn}, while \eqref{i: CH}$\implies$\eqref{i: DRS} is the content of \cite{DRS_hilbert}*{Theorem B}. Further, the implications \eqref{i: Klee}$\implies$\eqref{i: tiling} and \eqref{i: DRS}$\implies$\eqref{i: tiling} are obvious. Finally, \eqref{i: tiling}$\implies$\eqref{i: CH} is just a consequence of the previous corollary.
\end{proof}

\section{Tilings in \texorpdfstring{$c_0$}{c0}-like spaces}\label{sec: c0}
In this section we study the existence of tilings by balls in spaces of the form $\X\oplus_\infty c_0$ and $\X\oplus_\infty c_{00}$. Among our motivations is to show that the completeness assumption in \Cref{thm: fd reduction}\eqref{tiling sep} cannot be dropped (\Cref{rmk: noncomplete false}). Since in this section we don't need to consider degenerate balls, throughout this section by ball we always mean a (closed) ball of positive radius. 

Before the main result, we begin with a characterisation of those separable normed spaces $\X$ for which $\X\oplus_\infty c_0$ admits a tiling by balls.

\begin{proposition} Given a separable normed space $\X$, the space $\X\oplus_\infty c_0$ admits a tiling by balls if and only if $\X$ does.
\end{proposition}

\begin{proof} To begin with, suppose that $\B= \{B_\X(x_k,r_k)\}_{k=1}^\infty$, $x_k\in \X$, is a tiling of $\X$ by balls. Let $(z_j)_{j=1}^\infty$ be an injective enumeration of all sequences in $c_0$ that are $2\ZZ$-valued. Then
\[ \B'\coloneqq \{B\big( (x_k, r_k z_j), r_k\big) \}_{j,k=1}^\infty \]
is a tiling by balls of $\X\oplus_\infty c_0$. We first prove that these balls don't overlap. For this take indices $j,k,j',k'\in \N$ with $(j,k)\neq (j',k')$. Then, we have
\[ \|(x_k, r_k z_j)- (x_{k'}, r_{k'} z_{j'})\|= \max\{ \|x_k-x_{k'}\|_\X, \|r_k z_j- r_{k'} z_{j'}\|_{c_0} \}. \]
If $k\neq k'$, this quantity is at least $\|x_k-x_{k'}\|_\X \geq r_k+ r_{k'}$ (as $\B$ is non-overlapping). Otherwise, if $k= k'$, the above quantity equals $r_k \|z_j- z_{j'}\|_{c_0} \geq 2r_k$. Therefore, $\B'$ is non-overlapping. Next, we prove that $\B'$ is a covering of $\X\oplus_\infty c_0$. Take any $(x,y)\in \X\oplus_\infty c_0$ and find $k\in \N$ such that $x\in B_\X(x_k,r_k)$. Moreover, find $j\in \N$ such that $\frac{1}{r_k}y\in B_{c_0}(z_j,1)$, namely $y\in B_{c_0}(r_k z_j, r_k)$. Then, $(x,y)\in B\big((x_k, r_k z_j), r_k\big)$, and $\B$ is a tiling.

For the converse implication, suppose that $\B$ is a tiling by balls of $\X \oplus_\infty c_0$. By \Cref{fact: good intersection with balls}, up to translating each ball in $\B$ by the same vector, we can suppose that each $B\in \B$ that intersects $\X$ is such that $\X\cap \inte B\neq\emptyset$. Thus the family
\[ \B'\coloneqq \{B\cap \X\colon B\in \B\} \]
is a tiling of $\X$ by balls (note that $\X$ has the (GIB) property in $\X\oplus_\infty c_0$).
\end{proof}

We now move to the main result, which shows that $\X\oplus_\infty c_{00}$ can admit a tiling by balls even when $\X$ doesn't. The rough idea of the proof (present in the proof of \Cref{lemma: induction}) is to take balls in $\X$ and see them as the intersection of $\X$ with non-overlapping balls of $\X\oplus_\infty c_{00}$. For the easiest instance of this, given balls $B_\X(x_1,r_1)$ and $B_\X(x_2,r_2)$ in $\X$, they are the intersection of $\X$ with the balls $B_{\X\oplus_\infty \R}((x_1, r_1), r_1)$ and $B_{\X\oplus_\infty \R}((x_2, -r_2), r_2)$ of $\X\oplus_\infty \R$ and these latter balls do not overlap. In turn, the balls in $\X$ will be chosen via the following lemma, from \cite{DESOVEstar}.

\begin{lemma}[\cite{DESOVEstar}*{Lemma~2.11}]\label{lemma: DESOVE} Let $\Z$ be a finite-dimensional subspace of a normed space $\Y$ and $C$ be a closed subset of $\Z$. Then there exists a family $\F$ of balls of $\Y$ such that:
\begin{enumerate}[label={\rm (\alph*)},ref=\alph*]
    \item\label{i: star} $\F$ is star-finite;
    \item\label{i: F disjoint} $c(B)\in \Z$ and $B\cap C=\emptyset$ for each $B\in\F$;
    \item\label{i: F covers} $\Z\setminus C\subseteq \bigcup\F$;
    \item\label{i: singular} the singular points of $\F$ are contained in $C$.
\end{enumerate} 
\end{lemma}

We now give the main lemma. Here and in \Cref{thm: Y+c00} we indicate by $\pi_\Y$ the canonical projection from $\Y\oplus_\infty c_{00}$ onto $\Y$. We also consider $\Y$ as a subspace of $\Y\oplus_\infty c_{00}$, by identifying it with $\Y\times\{0\}$.

\begin{lemma}\label{lemma: induction} Let $\Z$ be a finite-dimensional subspace of a normed space $\Y$ and denote $\X\coloneqq \Y\oplus_\infty c_{00}$. Also, let $\C$ be a family of balls in $\X$ with the following properties:
\begin{enumerate}
    \item\label{i: HP closed} the set $\bigcup \C$ is closed in $\X$;
    \item\label{i: HP centre in Z} $\pi_\Y(c(B))\in \Z$ and $B\cap \Z\neq \emptyset$ for each $B\in \C$.
\end{enumerate}
 Then, there exists a family $\B$ of balls of $\X$ such that:
\begin{enumerate}[label={\rm (\arabic*)},ref=\arabic*]
    \item\label{i: non-overlap} $\B$ is star-finite and non-overlapping;
    \item\label{i: centre in Z} $\pi_\Y(c(B))\in \Z$ and $B\cap \Z \neq \emptyset$ for each $B\in \B$;
    \item\label{i: cover Z} $\Z\subseteq \bigcup (\B\cup\C)$;
    \item\label{i: disjoint from C} $B\cap \bigcup\C=\emptyset$, whenever $B\in\B$;
    \item\label{i: U closed} the set $\bigcup (\B\cup\C)$ is a closed set in $\X$.
\end{enumerate}
\end{lemma}

\begin{proof} Let us denote by $\n$ the norm of $\Y$, by $\n_\infty$ the canonical norm of $c_{00}$, and by $\nn\cdot$ the norm of $\X$. By \eqref{i: HP closed}, the set $C\coloneqq \Z\cap \bigcup \C$ is a closed subset of $\Z$, and thus an application of \Cref{lemma: DESOVE} gives us a family $\F$ of balls in $\Y$ that satisfy \eqref{i: star}--\eqref{i: singular}. Since in particular $c(B)\in\Z$ for each $B\in \F$, we see that $\{B\cap\Z\colon B\in \F\}$ is a star-finite collection of balls in $\Z$, thus $\F$ is at most countable (by \cite{DESOVEstar}*{Lemma 2.2}). Write
\[ \F=\{B_\Y(y_k,r_k)\}_{k=1}^\infty, \]
where $y_k\in \Z$ and $r_k>0$. We can now define the family $\B$. For each $k\in \N$, set
\[ x_k\coloneqq (y_k,r_ke_1+ \dots+ r_ke_{k-1} - r_ke_k)\in\X \]
and consider the family $\B$ of balls given by
\[ \B\coloneqq \{B_\X(x_k,r_k)\}_{k=1}^\infty. \]
Since $y_k\in \Z$ and $(y_k,0) \in B_\X(x_k,r_k)\cap \Z$, the validity of \eqref{i: centre in Z} is clear by construction. Moreover, when $k<n$,
\[ \nn{x_n- x_k}\geq r_n+ r_k, \]
as is seen by checking the coefficients of $e_k$; thus, $\B$ is non-overlapping. Next, the fact that $\|r_ke_1+ \dots+ r_ke_{k-1} - r_ke_k\|_\infty = r_k$ implies the equality
\begin{equation}\label{eq: ball in X and Y}
    B_\X(x_k,r_k)\cap \Y= B_\Y(y_k,r_k).
\end{equation}
This and property \eqref{i: F covers} of $\F$ imply \eqref{i: cover Z}.

We next prove that $\B$ is star-finite (thus obtaining \eqref{i: non-overlap}). If the balls $B_\X(x_k,r_k)$ and $B_\X(x_n,r_n)$ in $\B$ have non-empty intersection, then $\nn{x_n- x_k}\leq r_n+ r_k$, whence $\|y_n- y_k\|\leq r_n+ r_k$ as well. Thus, the corresponding balls $B_\Y(y_k,r_k)$ and $B_\Y(y_n,r_n)$ in $\F$ intersect as well. As $\F$ is star-finite, this implies the same for $\B$.

A similar argument also proves \eqref{i: disjoint from C}. Suppose by contradiction that a ball $B\in \C$ intersects some $B_\X(x_k,r_k)$. Let $r$ be the radius of $B$ and $(c,v)\in \Y\oplus_\infty c_{00}$ its centre. Notice that, by \eqref{i: HP centre in Z}, $c\in \Z$ and $B\cap \Z \neq \emptyset$; thus, $B_\Z(c,r)= B\cap \Z$. Our assumption gives $\nn{(c,v)-x_k}\leq r+ r_k$, whence $\|c-y_k\|\leq r+ r_k$. Thus, $B_\Z(c,r)$ intersects $B_\Z(y_k,r_k)$. Further, $B_\Z(c,r)= B\cap \Z \subseteq C$, while $B_\Z(y_k,r_k)\subseteq B_\Y(y_k,r_k)\subseteq \bigcup\F$, contradicting property \eqref{i: F disjoint} of $\F$.

It remains to prove \eqref{i: U closed}. For this, by \eqref{i: HP closed} it is enough to prove that $\overline{\bigcup\B}^\X \setminus \bigcup\B\subseteq \bigcup\C$. Thus, let $x=(y,w)\in \overline{\bigcup\B}^\X \setminus \bigcup\B$ and let $(B_j)_{j=1}^\infty\subseteq \B$ be an injective sequence such that $\dist_\X(x,B_j)\to 0$. Let $B_j= B_\X(x_{k_j}, r_{k_j})$. As $\B$ is star-finite, up to passing to a subsequence we can assume that the balls $B_j$ are mutually disjoint. Thus, an easy argument (as in the proof of \Cref{claim: radius to 0}) applied to the balls $B_\Z(y_{k_j}, r_{k_j})$ shows that $r_{k_j}\to 0$. Therefore
\[ \nn{x_{k_j}- (y,w)}\to 0. \]
However, by construction, $x_{k_j}=(y_{k_j},w_{k_j})$, where $\|w_{k_j}\|_\infty= r_{k_j}$. As a consequence, $y_{k_j}\to y$ and $w=0$. The first property yields that each neighbourhood of $y$ intersects infinitely members of $\F$, thus $y$ is a singular point for $\F$. By \eqref{i: singular}, it then follows that $y\in C$. Therefore, $x=(y,0)\in \bigcup\C$, and we are done.
\end{proof}

Finally, we move to the main result on this section. 

\begin{theorem}\label{thm: Y+c00} Let $\Y$ be a normed space of at most countable dimension. Then the normed space $\X\coloneqq \Y\oplus_\infty c_{00}$ has a star-finite tiling by balls.
\end{theorem}

\begin{remark}\label{rmk: noncomplete false} Taking $\Y$ such that $S_\Y$ admits a non-protectable point (say $\Y= \ell_2^n$ or $\Y= \ell_1^3$) shows that the assumption of completeness cannot be dispensed with in \Cref{thm: fd reduction}\eqref{tiling sep}, as $\Y$ has the (GIB) property in $\X= \Y\oplus_\infty c_{00}$.
\end{remark}

\begin{proof} Write $\Y= \spn\{y_n\}_{n=1}^\infty$ and, for $n\in\N$, define the subspaces $\Z_n\coloneqq \spn\{y_k, e_k\}_{k=1}^n$ and $\Y_n\coloneqq \spn\bigl(\Y\cup\{e_1, \dots ,e_n\}\bigr)$. Notice that, for each $n\in \N$, $\Z_n$ is a finite-dimensional subspace of $\Y_n$ and $\X= \Y_n\oplus_\infty \spn\{e_k\}_{k= n+ 1}^\infty$ is isometric to $\Y_n\oplus_\infty c_{00}$. We inductively define families $\B_n$ ($n\in\N$) of balls in $\X$, satisfying for each $n\in\N$ the following conditions:
\begin{enumerate}
    \item[($\mathrm P^1_n$)] $\B_n$ is star-finite and non-overlapping;
    \item[($\mathrm P^2_n$)] for each $B\in \B_n$, $\pi_{\Y_n}(c(B))\in \Z_n$ (\emph{i.e.}, $\pi_\Y(c(B))\in \spn \{y_1,\dots, y_n\}$) and $B\cap \Z_n\neq \emptyset$;
    \item[($\mathrm P^3_n$)] $\Z_n\subseteq \bigcup(\B_1\cup \dots\cup \B_n)$;
    \item[($\mathrm P^4_n$)] $\bigcup\B_n$ is disjoint from $\bigcup\B_k$ for all $k< n$;
    \item[($\mathrm P^5_n$)] the set $\bigcup(\B_1\cup \dots\cup \B_n)$ is closed in $\X$.
\end{enumerate}

In fact, an application of \Cref{lemma: induction} to the normed space $\Y_1$, the subspace $\Z_1$, and the family $\C=\emptyset$ gives us a family $\B_1$ of balls in $\X$ satisfying conditions ($\mathrm P^1_1$)--($\mathrm P^5_1$). Now, take $n\geq 2$ and assume that, for each $k\leq n-1$, we have already defined $\B_k$ such that properties ($\mathrm P^1_k$)--($\mathrm P^5_k$) hold. Define $\C=\B_1\cup\dots\cup \B_{n-1}$ and notice that, by ($\mathrm P^5_{n-1}$), the set $\bigcup\C$ is closed in $\X$. Also, if $B\in \C$, take $k\leq n-1$ such that $B\in \B_k$; then, by ($\mathrm P^2_k$) we have
\[ \pi_\Y(c(B))\in \spn \{y_1,\dots, y_k\}\subseteq \spn \{y_1,\dots, y_n\}. \]
Thus, the family $\C$ satisfies the assumptions of \Cref{lemma: induction} (applied to the subspace $\Z_n$ of $\Y_n$), hence the lemma yields us a family $\B_n$ that satisfies \eqref{i: non-overlap}--\eqref{i: U closed}. These conditions directly imply that ($\mathrm P^1_n$)--($\mathrm P^5_n$) hold, and conclude the induction step.
  
Finally, let $\B=\bigcup_{n\in\N}\B_n$. By property ($\mathrm P^3_n$), we immediately get that $\B$ is a covering. Further, ($\mathrm P^4_n$) implies that the sets $\bigcup \B_n$ ($n\in\N$) are pairwise disjoint; hence by property ($\mathrm P^1_n$) we obtain that $\B$ is a star-finite tiling of $\X$.
\end{proof}

\begin{remark} In particular, the normed space $c_{00}$ admits a star-finite tiling by balls. On the other hand, by \Cref{thm: fd reduction}\eqref{star-n-finite} it does not have any star-\textit{n}-finite tiling by balls, for any $n\in\N$ (notice that $\Y_n=\spn \{e_1,\dots, e_n\}$ has the (GIB) property in $c_{00}$). Therefore, even for tilings, star-finiteness is weaker than being star-$n$-finite for some $n\in\ \N$ (compare with \Cref{rmk: star not star-n}).
\end{remark}

\noindent\textbf{Acknowledgements.} We are grateful to the anonymous referees for their helpful reports.

\medskip
\noindent\textbf{Funding information.} The research of C.A.~De Bernardi has been partially supported by the GNAMPA (INdAM -- Istituto Nazionale di Alta Matematica) and by the MICINN project PID2020-112491GB-I00 (Spain). The research of T.~Russo and J.~Somaglia has been partially supported by the GNAMPA (INdAM -- Istituto Nazionale di Alta Matematica). The research of T.~Russo was funded in part by the Austrian Science Fund (FWF) 10.55776/PAT1653825. For open access purposes, the authors have applied a CC BY public copyright license to any author accepted manuscript version arising from this submission.


\end{document}